\documentclass[11pt,a4]{article}
\usepackage{amsmath, amsthm, amsfonts, amssymb}
\def\Z{{\mathbb Z}}
\def\R{{\mathbb R}}

\def\E{{\mathbb E}}
\def\P{{\mathbb P}}
\def\Poisson{\mathcal P}

\def\cC{{\mathcal C}}

\def\cN{{\mathcal N}}
\def\cP{{\mathcal P}}

\def\cK{{\mathcal K}}

\def\X{{\mathcal X}}
\def\cX{{\mathcal X}}

\def\bfN{\mathbf N}

\DeclareMathOperator{\supp}{supp}

\DeclareMathOperator{\Var}{Var}
\DeclareMathOperator{\Ex}{{\mathbb E}}
\DeclareMathOperator{\Pois}{{Pois}}

\newtheorem{thm}{Theorem}[section]

\newtheorem{lem}[thm]{Lemma}

\theoremstyle{definition}
\newtheorem{defn}[thm]{Definition}

\theoremstyle{remark}

\title{A remark on the convergence of Betti numbers in the thermodynamic regime}

\author{Trinh Khanh Duy}

\date{}
\begin{document}
\maketitle
\begin{abstract}
	The convergence of the expectations of Betti numbers of \v Cech complexes built on binomial point processes in the thermodynamic regime is established.  
	
	
\end{abstract}

\section{Terminologies and main results}
\begin{defn}[\v Cech complex]
	Let $\X = \{x_1, x_2, \dots, x_n\}$ be a collection of points in $\R^d$. The \v Cech complex $\cC(\X,r)$, for $r > 0$, is constructed as follows.
\begin{itemize}
	\item[(i)] The $0$-simplices (vertices) are the points in $\X$.
	\item[(ii)] A $k$-simplex $[x_{i_0}, \dots, x_{i_k}]$ is in $\cC(\X, r)$ if $\bigcap_{j = 0}^k B_{r/2} (x_{i_j}) \ne \emptyset$.
\end{itemize}
Here $B_r(x) = \{y \in \R^d : \|y - x\| \le r\}$ denotes a ball of radius $r$ and center $x$, and $\|x\|$ is the Euclidean norm of $x$. The \v Cech complex can be also constructed from an infinite collection of points.
\end{defn}

Let $X_1, X_2, \dots,$ be a sequence of i.i.d.~(independent identically distributed) random variables with common probability density function $f(x)$. Define the induced binomial point processes as $\cX_n = \{X_1, \dots, X_n\}$. The object here is the \v Cech complex $\cC(\cX_n, r_n)$ built on $\cX_n$, where the radius $r_n$ also varies with $n$. Denote by $\beta_k(\cK)$ the $k$th Betti number, or the rank of the $k$th homology group, of a simplicial complex $\cK$. The limiting behaviour of Betti numbers $\beta_k(\cC(\cX_n, r_n))$ in various regimes has been studied recently by many authors. See \cite{BH} for a brief survey. This aim of this paper is to refine a limit theorem in the thermodynamic regime, a regime that $n^{1/d} r_n \to r \in (0, \infty)$.

In the thermodynamic regime, the expectations of the $k$th Betti numbers, for $1 \le k \le d-1$, grow linearly in $n$, that is, $c_1 n \le \Ex [\beta_k(\cC(\cX_n, r_n))]  \le c_2 n$ as $n \to \infty$. After centralizing, the strong law of large number holds,
\[
	\frac{1}{n} \Big(\beta_k(\cC(\cX_n, r_n)) - \Ex [\beta_k(\cC(\cX_n, r_n))] \Big) \to 0 \text{ a.s.~as } n \to \infty,
\]
provided that the density function $f$ has compact, convex support and that on the support of $f$, it is bounded both below and above \cite[Theorem~4.6]{ysa}. A remaining problem is to describe the exact limiting behaviour of the expected values of the Betti numbers. This paper gives a solution to that problem. Note that the $0$th Betti number which counts connected components in a random geometric graph was completely described \cite[Chapter~13]{Penrose-book}.

Betti numbers are tightly related to the number of $j$-simplices in $\cC(\cX, r)$, denoted by $S_j(\cC(\cX,r))$ or simply by $S_j(\cX, r)$, which can be expressed as 
\[
	S_j(\cX, r) = \frac{1}{j + 1} \sum_{x \in \cX} \xi(x, \cX),
\]
where $\xi(x, \cX)$ is the number of $j$-simplices containing $x$. Note that $\xi(x, \cX)$ is a local function in the sense that it depends only on points near $x$. Then in the thermodynamic regime, the weak and strong laws of large numbers for $S_j(\cC(\cX_n, r_n))$ hold as a consequence of general results in \cite{Penrose-07, PY-03},
\[
	\frac{S_j(\cX_n, r_n)}{n} \to \hat S_j \text{ a.s.~as } n \to \infty.
\]
The limit $\hat S_j$ can be expressed explicitly. However, Betti numbers do not have expression like the above form, and hence those general results can not be applied.

To establish a limit theorem for Betti numbers, we exploit the following two properties. The first one is the nearly additive property of Betti numbers that was used in \cite{ysa} to study Betti numbers of the \v Cech complex built on stationary point processes. The second one is the property that binomial point processes behave locally like a homogeneous Poisson point process. The latter property is also a key tool to establish the law of large numbers for local geometric functionals \cite{Penrose-07, PY-03}.

Now let us get into more detail to state the main result of the paper. We begin with the definition of a homogeneous Poisson point processes. Let $\bfN$ be the set of all counting measures on $\R^d$ which are finite on any bounded Borel set and for which the measure of a point is at most $1$. Define $\cN$ as the $\sigma$-algebra generated by sets of the form
\[
	\{\mu \in \bfN : \mu(A) = k\},
\] 
where $A$ is a Borel set and $k$ is an integer. Then a point process $\Phi$ is a measurable mapping from some probability space into $(\bfN, \cN)$. For a Borel set $A$, let $\Phi(A)$ denote the number of points in $A$. By definition of the $\sigma$-algebra $\cN$,  $\Phi(A)$ is a random variable. A homogeneous Poisson point process is defined as follows. For some basic properties of point processes, see \cite{Meester-Roy}, for example.
\begin{defn}[Homogeneous Poisson point process] The point process $\cP$ is said to be a Poisson point process with density $\lambda > 0$ if 
\begin{itemize}
	\item[\rm(i)]	for disjoint Borel sets $A_1, \dots, A_k$, the random variables $\cP(A_1), \dots, \cP(A_k)$ are independent;
	\item[\rm(ii)]	for any bounded Borel set $A$, the number of points in $A$ has Poisson distribution with parameter $|A|$, $\cP(A) \sim \Pois(|A|)$, that is,
	\[
		\P(\cP(A) = k) = e^{- \lambda |A|} \frac{\lambda^k |A|^k}{k!}, \quad k = 0,1, \dots,
	\] 
where $|A|$ denotes the Lebesgue measure of $A$.
\end{itemize}
\end{defn}

For homogeneous Poisson point processes, the following law of large numbers for Betti numbers was established in \cite{ysa}. Let $\cP(\lambda)$ be a homogeneous Poisson point process on $\R^d$ with density $\lambda > 0$. Denote by $\cP_{A}(\lambda)$ the restriction of $\cP(\lambda)$ on a Borel set $A$. For a window of the form $W_L = [-\frac {L^{1/d}}2, \frac{L^{1/d}}2)^d$, we write $\cP_{L}(\lambda)$ instead of $\cP_{W_L}(\lambda)$. For $\lambda = 0$, we mean a trivial point process with no point and all functions are assumed to be zero at $\lambda = 0$. Then for $1 \le k \le d - 1$, there is a constant $\hat \beta_k (\lambda, r)$ such that   \cite[Theorem~3.5]{ysa},
\[
	\frac{\beta_k(\cC(\cP_L(\lambda), r)) } {L} \to \hat \beta_k(\lambda, r) \text{ a.s.~as } L \to \infty.
\]
Now we can state our main result.

\begin{thm}\label{thm:binomial-process}
	Assume that the common probability density function $f(x)$ has bounded support, is bounded and Riemann integrable. Then as $n \to \infty$ with $n^{1/d} r_n =  r \in (0, \infty)$, 
	\[
		\frac{\E[\beta_k(\cC(\cX_n, r_n))]}{n} \to \int_\R \hat\beta_k(f(x), r) dx .	\]
\end{thm}

For the proof, we need a Poissonized version of the binomial processes. Let $N_n$ be a random variable which is independent of $\{X_n\}_{n \ge 1}$ and has Poisson distribution with parameter $n$. Let 
\[
	\bar \cP_n = \{X_1, X_2, \dots, X_{N_n}\}.
\]
Then $\bar \cP_n$ becomes a non-homogeneous Poisson point process with intensity function $n f(x)$. Here a non-homogeneous Poisson point process is defined as follows.

\begin{defn}[Non-homogeneous Poisson point process] Let $f(x) \ge 0$ be a locally integrable function on $\R^d$. The point process $\cP$ is said to be a (non-homogeneous) Poisson point process with intensity function $f(x)$ if 
\begin{itemize}
	\item[\rm(i)]	for mutually disjoint Borel sets $A_1, \dots, A_k$, the random variables $\cP(A_1), \dots, \cP(A_k)$ are mutually independent;
	\item[\rm(ii)]	for any bounded Borel set $A$, $\cP(A)\sim \Pois(\int_A f(x) dx)$.
\end{itemize}
\end{defn}

As proved later, Theorem~\ref{thm:binomial-process} is equivalent to the following result.
\begin{thm}\label{thm:non-homo-Poisson}
Assume that the common probability density function $f(x)$ has bounded support, is bounded and Riemann integrable. Then as $n \to \infty$ with $n^{1/d} r_n =  r \in (0, \infty)$, 
	\[
		\frac{\E[\beta_k(\cC(\bar \cP_n, r_n))]}{n} \to \int_\R \hat\beta_k(f(x), r) dx .	
	\]
\end{thm}

\section{Proofs of main theorems}
We will use the following two important properties of Poisson point processes.
Denote by $\cP(f(x))$ the non-homogeneous Poisson point process with intensity function $f(x)$. \begin{itemize}
	\item[(i)] Scaling property. For any $\theta > 0$ and $t \in \R^d$,
		\[
			\theta ( \cP(f(x)) - t) \overset{d}{=} \cP(\theta^{-d} f(t + \theta^{-1}x)),
		\]
where `$\overset{d}{=}$' denotes the equality in distribution.
	In particular, $\theta (\cP(\lambda) - t) \overset{d}{=} \cP(\theta^{-d} \lambda)$.
	\item[(ii)] Coupling property. Let $\cP(g(x))$ be a Poisson point process with intensity function $g(x)$ which is independent of $\cP(f(x))$. Then 
\[
	\cP(f(x)) + \cP(g(x)) \overset{d}{=} \cP(f(x) + g(x)).
\]
Here `$+$' means the superposition of two point processes.
\end{itemize}

We begin with a result for the simplices counting function.
\begin{lem}[cf.~{\cite[Lemma~3.2]{ysa}}]\label{lem:simplices-counting}
Let $S_j(\lambda, r ; L)$ be the number of $j$-simplices in $\cC(\cP_L(\lambda), r)$. Then for fixed $r > 0$, 
	\[
		\frac{\Ex[S_j(\lambda, r; L)]}{L} \to \hat S_j(\lambda, r) \text{ as } L \to \infty, \text{ uniformly for $0< \lambda \le \Lambda$}.
	\]
In addition, for fixed $r$, the limit $\hat S_j(\lambda, r)$ is a continuous function of $\lambda$ on $[0, \infty)$.
\end{lem}
\begin{proof}
For convenience, let $A_l(\lambda) := S_j(\lambda, r; l^d) = S_j( \cC(\cP_{V_l}(\lambda), r) )$, where $V_l = [-\frac{l}{2}, \frac{l}{2})^d$.  Our aim now is to show that 
\[
	\frac{\Ex[A_l(\lambda)]}{l^d} \text{ uniformly converges as } l \to \infty,
\]
and that $\Ex[A_l(\lambda)]$ is continuous for $\lambda \in [0, \infty)$. Let us first show the continuity of $\Ex[A_l(\lambda)]$. For $0 \le \lambda < \mu$, we use the coupling $\cP(\mu) = \cP(\lambda) + \cP(\mu - \lambda)$. Here $\cP(\lambda)$ and $\cP(\mu - \lambda)$ are two independent Poisson point processes with density $\lambda$ and $(\mu - \lambda)$, respectively. Let $N_\lambda$ (resp.~$N_{\mu; \lambda}$) be the number of points of $\cP(\lambda)$ (resp.~$\cP(\mu- \lambda)$) in $V_l$, which has Poisson distribution with parameter $\lambda l^d$ (resp.~$(\mu - \lambda)l^d$). 
Then the continuity follows from a trivial estimate
\begin{align*}
	0 \le A_l(\mu) - A_l(\lambda) 
	\le N_{\mu; \lambda} (N_{\mu; \lambda} + N_\lambda)^j.
\end{align*}

Next, we show the uniform convergence. The proof here is similar to that of the pointwise convergence (\cite[Lemma~3.2]{ysa}). Define the function 
\[
	h(\cP(\lambda)) := \frac{1}{j + 1} \sum_{x \in \cP_1(\lambda)} \#[j \text{-simplices in $\cC(\cP(\lambda), r)$ containing $x$}].
\]
Then for $l>2r + 1$,
\[
	\sum_{z \in \Z^d \cap V_{l -2r -1}} h(\cP(\lambda) - z) \le A_{l}(\lambda) \le \sum_{z \in \Z^d \cap V_{l + 2r + 1}} h(\cP(\lambda) - z).
\]
Consequently, by the stationality of the Poisson point process $\cP(\lambda)$,
\[
	(l - 2r - 2)^d \Ex [h(\cP(\lambda))] \le \Ex[A_{l}(\lambda)] \le (l+2r + 2)^d \Ex [h(\cP(\lambda))].
\]
Note that $\E[h(\cP(\lambda))]$ is non-decreasing in $\lambda$ and for any $\lambda > 0$,
\[
	\Ex [h(\cP(\lambda))] \le \Ex [\cP(\lambda; V_{1 + 2r})^{j + 1}] < \infty.
\]
Here $\cP(\lambda; V_{1 + 2r})$ is the number of points of $\cP(\lambda)$ in $V_{1+2r}$.
Therefore uniformly for $0\le \lambda \le \Lambda$, 
\[
	\frac{\Ex[A_{l}(\lambda)]}{ l^d} \to \E[h(\cP(\lambda))] \text{ as } l \to \infty.
\]
The proof is complete.
\end{proof}

The following estimate for Betti numbers is a key tool to derive the convergence of Betti numbers from that of simplices counting functions. Recall that $\beta_k(\cK)$ denotes the $k$th Betti number of the simplicial complex $\cK$.
\begin{lem}[{\cite[Lemma~2.2]{ysa}}]
\label{lem:Betti-estimate}
	Let $\cK_1, \cK_2$ be two finite simplicial complexes such that $\cK_1 \subset \cK_2$. Then for every $k \ge 1$,
	\[
		\left| \beta_k(\cK_1) - \beta_k (\cK_2) \right| \le \sum_{j = k}^{k + 1} \#\{j \text{-simplices in $\cK_2 \setminus \cK_1$}\}.
	\]
\end{lem}

For the sake of simplicity, we denote by $\beta_k(\lambda, r; L)$ the $k$th Betti number of the \v Cech complex $\cC(\cP_{W_L} (\lambda), r)$, where $W_L$ is any rectangle of the form $x + [-\frac{L^{1/d}}{2}, \frac{L^{1/d}}{2})^d$.

\begin{lem}
For fixed $r>0$, uniformly for $0 \le \lambda \le \Lambda$,
\[
	\frac{\Ex[ \beta_k (\lambda, r; L)]}{L} \to \hat \beta_k(\lambda, r) \text{ as } L \to \infty.
\]
The limit $\hat \beta_k(\lambda, r)$ has the following scaling property,
\[
	\hat \beta_k(\lambda, r) = \frac{1}{\theta}\hat \beta_k( \lambda \theta, \frac{r}{\theta^{1/d}} ), \text{ for any $\theta > 0$}.
\] 
In particular, $\hat \beta_k(\lambda, r) = \lambda \hat \beta_k (1, \lambda^{1/d} r)$ is a continuous function in both $\lambda$ and $r$, and $\hat\beta(\lambda, r) > 0$, if $\lambda > 0$ and $r> 0$.
\end{lem}

\begin{proof}
	For fixed $r > 0$ and fixed $\lambda > 0$, the convergence of the expectations of Betti numbers was shown in \cite[Lemma~3.3]{ysa}. The positivity is a consequence of \cite[Theorem~4.2]{ya}.  Here we will show the uniform convergence for $0\le \lambda \le \Lambda$. We use the following criterion for the uniform convergence on a compact set, which is related to the Arzel\`a--Ascoli theorem.  The sequence of continuous functions $\{a_L(\lambda)\}_{L > 0}$ converges uniformly on $[0, \Lambda]$ if and only if it converges pointwise and is equicontinuous, that is, for any $\varepsilon > 0$, there are $\delta > 0$ and $L_0 > 0$ such that 
\[
	|a_L(\lambda_1) - a_L(\lambda_2) |  < \varepsilon \text{ for all $\lambda_1, \lambda_2 \in [0, \Lambda]$,  $|\lambda_1 - \lambda_2| < \delta$, and all $L > L_0$.}
\]
	
Our task now is to show that the sequence 
$\{L^{-1} \Ex[ \beta_k (\lambda, r; L)]\}$ is equicontinuous. Let $\lambda < \mu$. By using the coupling $\cP(\mu) = \cP(\lambda) + \cP(\mu - \lambda)$, the \v Cech complex $\cC(\cP_L(\lambda), r)$ becomes a sub-complex of $\cC(\cP_L(\mu), r)$. Thus, by Lemma~\ref{lem:Betti-estimate}, 
\begin{align*}
	|\beta_{k}(\mu, r ; L) - \beta_{k}(\lambda, r ; L)| &\le \sum_{j = k}^{k + 1} \#\left \{j\text{-simplices in }C(\cP_{L}(\mu), r) \setminus C(\cP_{L}(\lambda), r)  \right\}  \\
	&=\sum_{j = k}^{k + 1} (S_{j}(\mu, r; L) - S_{j}(\lambda, r; L)).
\end{align*}
Therefore 
\[
	\left| \frac{\Ex [ \beta_{k}(\mu, r ; L)]}{L} - \frac{\Ex [ \beta_{k}(\lambda, r; L)]} {L} \right| \le \sum_{j = k}^{k + 1} \left( \frac{\Ex [S_{j}(\mu, r ; L)]}{L} - \frac{\Ex [ S_{j}(\lambda, r; L)]}{L} \right).
\]
The sequence $\{L^{-1}\E[S_j(\lambda, r; L)]\}$ converges uniformly on $[0, \Lambda]$ by Lemma~\ref{lem:simplices-counting}, and hence, is equicontinuous, which then implies the equicontinuity of the sequence 
$\{L^{-1} \Ex[ \beta_k (\lambda, r; L)]\}$.

By observing that $\theta^{-1/d} \cP(\lambda)$ has the same distribution with $\cP(\lambda \theta)$, we obtain the scaling property of $\hat \beta_k(\lambda, r)$. It then follows from the scaling property that $\hat \beta_k(\lambda, r)$ is continuous in both $\lambda$ and $r$.
The lemma is proved.
\end{proof}

Let us now consider the scaled Poissonized version $\cP_n = \{n^{1/d}X_1, n^{1/d} X_2, \dots, n^{1/d}X_{N_n}\}$. Recall that $N_n$ is independent of $\{X_n\}$ and has Poisson distribution with parameter $n$. Then $\cP_n = n^{1/d}\bar \cP_n$ is a non-homogeneous Poisson point process with the intensity function $f_n(x) := f(x/n^{1/d})$. It is clear that $\cC(\cP_n, r) = \cC(\bar \cP_n, r_n)$ because $n^{1/d}r_n = r$. Thus  Theorem~\ref{thm:non-homo-Poisson} can be rewritten as follows.
\begin{thm}\label{thm:non-homo-Poisson-scaled}
Assume that the common probability density function $f(x)$ has bounded support, is bounded and Riemann integrable. Then for fixed $r > 0$, as $n \to \infty$,
\[
	\frac{\Ex [\beta_k (\cC(\cP_n, r))]}{n} \to \int_\R \hat\beta_k(f(x), r) dx = \int_\R \hat \beta_k(1, {f(x)^{1/d}} r) f(x) dx .
\]
\end{thm}

\begin{lem}\label{lem:difference-two-Poisson}
Assume that $f(x), g(x) \le \Lambda$ in $W_L$, where $W_L \subset \R^d$ is a set of volume $L$. Then there exists a constant $c= c(k,\Lambda L)$ such that 
\[
	\Big| \Ex [\beta_k(\cC(\cP_{W_L}(f(x)), r))] - \Ex [\beta_k(\cC(\cP_{W_L}(g(x)), r))] \Big| \le c \int_{W_L}|f(x) - g(x)|dx.
\]
\end{lem}
\begin{proof}
By considering $f(x) := f(x)|_{W_L}$ and $g(x) := g(x)|_{W_L}$, we omit the subscript $W_L$ in formulae. Let $h(x) = \max \{f(x), g(x) \}$. A key idea here is the following coupling 
\[
	\Poisson(h(x)) = \Poisson(f(x)) + \Poisson(h(x) - f(x)).
\] 
Let $t = \int (h(x) - f(x)) dx = \int (g(x) - f(x))^{+} dx$ and $N_t$ be the number of points of $\cP(h(x)- f(x))$ in $W_L$. Then $N_t$ has Poisson distribution with parameter $t$. The total number of points of $\cP(h(x))$ is bounded by $N_t + N_{\Lambda l - t}$, where $N_{\Lambda L -t}$ has Poisson distribution with parameter $(\Lambda L - t)$ which is independent of $N_t$. It now follows from Lemma~\ref{lem:Betti-estimate} that 
	\begin{align*}
		\Big|\beta_k(\cC(\cP(f(x)), r)) - \beta_k(\cC(\cP(h(x)), r)) \Big| &\le \sum_{j = k}^{k + 1} S_j \Big(\cC(\cP(h(x)), r) \setminus C(\cP(f(x)), r) \Big) \\
		&\le 2N_t (N_t + N_{\Lambda L - t})^{k+ 1},
	\end{align*}
and hence, 
	\[
		\Big|\Ex[\beta_k(\cC(\cP(f(x)), r))] - \Ex[\beta_k(\cC(\cP(h(x)), r)) ]\Big| \le 2\E[ N_t (N_t + N_{\Lambda L - t})^{k+1}].
	\]
The right hand side is a polynomial of $t$ whose smallest order is $1$ and note that  $t \le \Lambda L$, thus it is bounded by $c(k, \Lambda L) t$, where the constant $c(k, \Lambda L)$ depends only on $k, \Lambda L$, namely we have 
\[
	\Big|\Ex[\beta_k(\cC(\cP(f(x)), r))] - \Ex[\beta_k(\cC(\cP(h(x)), r)) ]\Big| \le 	c \int (g(x) - f(x))^{+}dx.
\]

An analogous estimate holds when we compare the $k$th Betti number of  $\cC(\cP(g(x)), r)$ and $\cC(\cP(h(x)), r)$. The proof is complete.
\end{proof}

\begin{proof}[Proof of Theorem~{\rm\ref{thm:non-homo-Poisson-scaled}}]
Let $S$ be the support of $f$ and $\Lambda := \sup f(x)$. 
Divide $\R^d$ according to the lattice $(L/n)^{1/d}\Z^d$ and let $\{C_i\}$ be the cubes which intersect with S. Since we also consider the Poisson point process with density $0$, we may assume that $S = \cup_i C_i$. 

Let $W_i$ be the image of $C_i$ under the map $x \mapsto n^{1/d}x$. Then $W_i$ is a cube of length $L^{1/d}$. Let $\beta_k(W_i, r)$ be the $k$th Betti number of $\cC(\cP_n|_{W_i}, r)$. 
We now compare the $k$th Betti number of $\cC(\cP_n, r)$ and that of $\cup_{i} \cC(\cP_n|_{W_i}, r)$ by using Lemma~\ref{lem:Betti-estimate},
\begin{align}
	\Big| \beta_k(\cC(\cP_n, r)) - \beta_k(\bigcup_i\cC(\cP_n|_{W_i}, r)) \Big| &\le \sum_{j = k}^{ k + 1} S_j\Big(\cC(\cP_n, r) \setminus \bigcup_i\cC(\cP_n|_{W_i}, r) \Big) \nonumber\\
	&\le \sum_{j = k}^{k + 1} S_j( \cP_n, r; \cup_i (\partial W_i)^{(r)}). \label{boundary}
\end{align}
Here $S_j(\cP_n, r; A)$ is the number of $j$-simplices in $\cC(\cP_n, r)$ which has a vertex in $A$, $\partial A$ denotes the boundary of the set $A$ and $A^{(r)}$ is the set of points with distance at most $r$ from $A$. The second inequality holds because any simplex in $\cC(\cP_n, r) \setminus \cup_i\cC(\cP_n|_{W_i}, r)$ must have a vertex in $\cup_i (\partial W_i)^{(r)}$.

Finally, by the coupling $\cP(\Lambda) = \cP_n + \cP(\Lambda - f(x/n^{1/d}))$, it follows that for any bounded Borel set $A$,
\[
	\Ex [S_j(\cP_n, r; A)] \le \Ex [S_j(\cP(\Lambda), r ; A)]  \le \Ex[ \sum_{x \in \cP(\Lambda) \cap A} \cP(\Lambda; B_r(x))^{j}]=: \mu_{\Lambda, r, j} (A) < \infty.
\]
Here $\mu_{\Lambda, r, j}$ becomes a translation invariant measure on $\R^d$ which is finite on bounded Borel sets. Thus $\mu_{\Lambda, r, j} (A) = c(\Lambda, r, j) |A|$ for some constant $c(\Lambda, r, j)$ depending only on $\Lambda, r$ and $j$. Now by taking the expectation in \eqref{boundary}, we get 
\begin{eqnarray*}
	&\Big| \Ex [\beta_k(\cC(\cP_n, r))] - \sum_i \Ex [\beta_k(\cC(\cP_n|_{W_i}, r)) ] \Big| \nonumber\\ 
	&\le c  \sum_ i  |(\partial W_i)^{(r)}| \le c' \frac{n |S|}{L}  L^{(d-1)/d} = c' \frac{n |S|}{L^{1/d}},
\end{eqnarray*}
where $c$ and $c'$ are constants which do not depend on $n$ and $L$. Therefore,
\begin{equation}\label{Betti-union}
	\limsup_{n \to \infty} \Big| \frac{\Ex [\beta_k(\cC(\cP_n, r))]}{n} - \frac{1}{n} \sum_i \Ex [\beta_k(W_i, r) ] \Big|  \le c' \frac{ |S|}{L^{1/d}}.
\end{equation}

Let $f^{*}_i := \sup_{x \in C_i} f(x)$ and $\beta_k(f_i^*, r)$ be the $k$th Betti number of the \v Cech complex built on a homogeneous Poisson point process $\cP_{W_i}(f_i^*)$ with density $f_i^*$ restricted on $W_i$. Then by Lemma~\ref{lem:difference-two-Poisson},
\begin{align*}
	\Big|\Ex [\beta_k(W_i, r)] - \Ex [\beta_k(f^{*}_i, r)] \Big| &\le c(k, \Lambda L) \int_{W_i}{(f^{*}_i - f(x/n^{1/d})) dx} \\
	&= c(k, \Lambda L) n \int_{C_i} (f^{*}_i - f(x)) dx.
\end{align*}
Here $c(k, \Lambda L)$ is a constant depending only on $k$ and $\Lambda L$.
Consequently,
\[
	\left|  {\frac 1n \sum_{i} \Ex[ \beta_k(W_i, r)] - \frac 1n \sum_{i} \Ex [\beta_k(f^{*}_i, r)]}\right| \le c(k, \Lambda L) \sum_i \int_{C_i} (f_i^{*} - f(x)) dx \to 0 \text{ as } n \to \infty,
\]
because the function $f(x)$ is assumed to be Riemann integrable.

Next by comparing $\Ex [\beta_k(f_i^*, r)]$ with the limit $\hat \beta_k(\lambda, r)$, we get
\begin{align*}
	\left| {\frac 1n \sum_{i} \Ex [\beta_k(f^{*}_i, r)] - \frac{L}{n} \sum_i  \hat \beta_k(f^*_i, r)}\right| 	&\le \frac{L}{n} \#\{C_i \} \sup_{0 \le \lambda \le \Lambda} \left|\frac{\Ex [\beta_k(\lambda, L)]}{L} - \hat \beta_k(\lambda, r)\right|  \\
	&= |S| \sup_{0 \le \lambda \le \Lambda} \left|\frac{\Ex [ \beta_k(\lambda, L)]}{L} - \hat \beta_k(\lambda, r)\right|.
\end{align*}
Note that for fixed $L$, as $n \to \infty$, 
\[
	\sum_i \hat \beta_k(f^*_i, r) \frac Ln \to \int_S \hat \beta_k(f(x), r) dx.
\]
Therefore  
\begin{equation}\label{fix-l}
	\limsup_{n \to \infty} \left| \frac{1}{n} \sum_i \Ex [\beta_k(W_i, r)] - \int_S \hat \beta_k(f(x), r) dx \right| \le |S| \sup_{ 0 \le \lambda \le \Lambda} \left|\frac{\Ex [\beta_k(\lambda, L)]}{L} - \hat \beta_k(\lambda, r)\right|.
\end{equation}
Combining two estimates \eqref{Betti-union} and \eqref{fix-l} and then let $L \to \infty$, we get the desired result. The proof is complete.
\end{proof}

The result for binomial point processes will follow from Theorem~\ref{thm:non-homo-Poisson} and the following result.
\begin{lem}
As $n \to \infty$,
\[
	\left| \frac{\Ex[\beta_k(\cC(\bar \cP_n, r_n ))]}{n} - \frac{\E[\beta_k(\cC(\cX_n, r_n))]}{n} \right| \to 0.
\]
\end{lem}
\begin{proof}
 By Lemma~\ref{lem:Betti-estimate} again, we have,
\[
	\Big| \beta_k(\cC(\bar \cP_n, r_n)) - \beta_k(\cC(\cX_n, r_n))  \Big| \le \sum_{j = k}^{k + 1}  \Big|S_j(\cC(\bar \cP_n, r_n)) - S_j (\cC(\cX_n, r_n)) \Big|.
\]
The right hand side, divided by $n$, converges to $0$ as a consequence of general results in \cite{Penrose-07, PY-03} applied to $S_j$. Here we will give an easy proof. 

For any $m$, let 
\[
	S_j(m, n) = |S_j(\cC(\cX_m, r_n)) - S_j (\cC(\cX_n, r_n)) |.
\]
Since the probability density function $f(x)$ is bounded, in the regime that $n r_n^d \to r^d$,  the probability that $\{ X_1 \in B_x({r_n}) \}$ is bounded by 
\[
	\P(X_1  \in B_{x} (r_n)) \le \frac{c}{n}, 
\]
for some constant $c$ which does not depend on $n$.

For $m > n \ge j$, since each $j$-simplices in $\cC(\cX_m, r_n) \setminus \cC(\cX_n, r_n)$ must contain at least one vertex in $\{X_{n + 1}, \dots, X_m\}$, we have 
\begin{align*}
	\Ex[S_j(m, n)] &\le (m - n) \Ex [\#\{j \text{-simplices in $\cC(\cX_m, r_n)$ containing $X_m$}\}] \\
	&\le (m-n) \binom mj \P( X_1 \in B_{X_m} (r_n) , \dots, X_j \in B_{X_m}(r_n))  \\
	&\le (m - n) \frac{m !}{j! (m - j)!} \left(  \frac{c}{n} \right)^j \\
	&\le c_1  (m - n) \left( \frac{m}{n} \right)^j .
\end{align*}
When $j \le m < n$, we change the role of $m$ and $n$ to get 
\[
	\Ex[S_j(m, n)] \le (n - m) \binom n j  \left(  \frac{c}{n} \right)^j \le c_2  (n - m).
\]
Combining two estimates, we have 
\[
	\Ex[S_j(m, n)] \le c_3  |m - n| \left[1 + \left( \frac{m}{n} \right)^j \right].
\]
Therefore,
\begin{align*}
	\Ex \left[\left|S_j(\cC(\bar \cP_n, r_n)) - S_j (\cC(\cX_n, r_n)) \right| \right] &\le c_3 \Ex \left[ |N_n - n| \left( 1 + \frac{(N_n)^j}{n^j} \right) \right] \\
	&\le c_3 \Ex[(N_n - n)^2]^{1/2} \Ex \bigg[ \left( 1 + \frac{(N_n)^j}{n^j} \right)^2 \bigg]^{1/2}.
\end{align*}
Here in the last inequality, we have used the Cauchy--Schwarz inequality. Note that $\E[(N_n)^j]$ is a polynomial in $n$ of degree $j$. Thus the second factor in the above estimate remains bounded as $n \to \infty$. Note also that
\[
	\Ex[(N_n - n)^2] = \Var[N_n]  = n.
\]
Therefore, 
\[
	\frac{\Ex \left[\left|S_j(\cC(\bar \cP_n, r_n)) - S_j (\cC(\cX_n, r_n)) \right| \right]}{n} \le \frac{c_4}{n^{1/2}} \to 0 \text{ as } n \to \infty.
\]
The theorem is proved.
\end{proof}

\section{Concluding remarks}

Together with the law of large numbers in \cite{ysa}, we have the following result. Assume that the support of $f$ is compact and convex and that 
\[
	0 < \inf_{x \in \supp f} \le \sup_{x \in \supp f} f(x) < \infty.
\]
Assume further that $f$ is Riemann integrable. 
Then for $1 \le k \le d - 1$,
\[
	\frac{\beta_k(\cC(\cX_n, r_n))}{n} \to \int_\R \hat \beta_k(f(x), r) dx \text{ a.s.~as } n \to \infty.
\]
A result for the Vietoris-Rips complex also holds.
%

\begin{tabular}{l}
Trinh Khanh Duy \\
Institute of Mathematics for Industry \\
Kyushu University\\
Fukuoka 819-0395, Japan \\
e-mail: trinh@imi.kyushu-u.ac.jp 
\end{tabular}

\end{document}